\documentclass[a4pape,reqno,12pt,oneside]{amsart}
\usepackage{a4wide}
\usepackage[english]{babel}    
 
\usepackage{cite}
\usepackage{comment}
\usepackage{graphics}
\usepackage{verbatim}
\usepackage{subfigure}
\usepackage{graphicx}
\usepackage{epsfig}
\usepackage{hyperref}
\hypersetup{colorlinks,
linkcolor=blue,
filecolor=green,
urlcolor=blue,
citecolor=blue}
\usepackage{enumitem} 
\usepackage{framed}
\usepackage{psfrag}
\usepackage{tikz}
\usepackage{tkz-berge}
\usetikzlibrary{arrows,snakes,backgrounds,trees}
\usetikzlibrary{matrix,decorations.pathreplacing,positioning}
\usepackage{stmaryrd}
\usepackage[centertags]{amsmath}
\usepackage{indentfirst,amsfonts,amssymb,amsthm,newlfont}

\def\A{\mathcal{A}}

\newtheorem{lemma}{Lemma}[section]
\newtheorem{corollary}[lemma]{Corollary}
\newtheorem{theorem}[lemma]{Theorem}
\newtheorem{prop}[lemma]{Proposition}

\newtheorem{definition}[lemma]{Definition}
\newtheorem{example}[lemma]{Example}

\begin{document}

\subjclass[2010]{05C25, 17D92, 05C81}
\keywords{Evolution Algebra, Genetic Algebra, Isomorphism, Graph, Random Walk} 


\title[Singular graphs with non-isomorphic evolution algebras]{On some singular graphs with non-isomorphic associated evolution algebras}

\author[Paula Cadavid]{Paula Cadavid}
\address{Paula Cadavid: Departamento de Matemática - Universidade Federal Rural de Pernambuco, Rua Dom Manuel de Medeiros, s/n, Dois Irm\~aos, CEP 52171-900, Recife/PE, Brazil}
\email{paula.cadavid@ufrpe.br}

\author[Mary Luz Rodi\~no Montoya]{Mary Luz Rodi\~no Montoya}
\address{Mary Luz Rodi\~no Montoya: Instituto de Matem\'aticas - Universidad de Antioquia, Calle 67 N$^{\circ}$ 53-108, Medell\'in, Colombia}
\email{mary.rodino@udea.edu.co}

\author[Pablo M. Rodriguez ]{Pablo M. Rodriguez}
\address{Pablo M. Rodriguez: Centro de ciências Exatas e da Natureza - Universidade Federal de Pernambuco, Av. Prof. Moraes Rego, 1235, Cidade Universit\'aria, Recife/PE, Brazil.}
\email{pablo@de.ufpe.br}

\author[Sebastian J. Vidal]{Sebastian J. Vidal}
\address{Sebastian J. Vidal: Departamento de Matem\'atica, Facultad de Ingenier\'ia, Universidad Nacional de la Patagonia ``San Juan Bosco'', Km 4, CP 9000, Comodoro Rivadavia, Chubut, Argentina.}
\email{svidal@unpata.edu.ar}

\begin{abstract}

A connected graph can be associated with two distinct evolution algebras. In the first case, the structural matrix is the adjacency matrix of the graph itself. In the second case, the structural matrix is the transition probabilities matrix of the symmetric random walk on the same graph. It is well-known that, for a non-singular graph, both evolution algebras are isomorphic if, and only if, the graph is regular or biregular. Moreover, through examples and partial results, it has been conjectured that the same result remains true for singular graphs. The purpose of this work is to provide new examples supporting this conjecture. 

\end{abstract}
\maketitle

\section{Introduction}

In this work, we advance the study of evolution algebras associated with graphs. These genetic algebras, introduced by \cite{tian,tv}, are defined as follows.

\begin{definition}\label{def:evolalg}
Let $\A:=(\A,\cdot\,)$ be an algebra over a field $\mathbb{K}$. We say that $\A$ is an evolution algebra if it admits a basis $S:=\{e_i \}_{i \in \Lambda}$, such that
\begin{equation}\label{eq:ea}
e_i \cdot e_i = \sum_{k\in \Lambda} c_{ik} e_k, \text{for any }i\in \Lambda, \text{ and }
e_i \cdot e_j =0, \text{if }i\neq j.
\end{equation} 
\end{definition}

The scalars $c_{ik}\in \mathbb{K}$ are referred to as the structure constants of the algebra $\mathcal{A}$ relative to the basis $S$. A basis $S$ that satisfies \eqref{eq:ea} is known as a natural basis of the algebra $\mathcal{A}$. We point out that this definition assumes that $\Lambda$ is a Hamel basis; which implies that the sum in \eqref{eq:ea} can have only a finite number of nonzero terms. Whether $c_{ik} \in [0,1]$, for any $i,k\in \Lambda$, and $\sum_{k\in \Lambda} c_{ik}=1$, for any $i\in \Lambda$, then $\A$ is called a Markov evolution algebra. This name is due to a correspondence between $\A$ and a discrete-time Markov chain $(X_n)_{n\geq 0}$ with state space $\{x_1,x_2,\ldots,x_n,\ldots\}$ and transition probabilities given by $\nonumber c_{ik}:=\mathbb{P}(X_{n+1}=x_k|X_{n}=x_i),$ for $i,k\in \mathbb{N}_0$, and for any $n\in \mathbb{N}$, where $\mathbb{N}_0:=\mathbb{N}\setminus \{0\}$. This correspondence is maybe one of the most interesting of the theory. Each state of the Markov chain is identified with a generator of $S$. The first reference discussing the interplay between evolution algebras and Markov chains is \cite[Chapter 4]{tian}, where many well-known results coming from Markov chains are stated in the language of Markov evolution algebras. See also \cite{irene,SPP-indian} and the references therein for a review of recent results related Markov evolution algebras. In addition we recommend \cite{ceballosetal} for a recent review of general results about evolution algebras.

In this work we focus on the connection of evolution algebras and graphs. The first definitions of evolution algebras associated to graphs have been formalized by \cite{tian} and studied later by \cite{PMP,PMP1,PMP2,reis,nunez1,nunez2}. We refer the reader also to \cite{PMPY,Elduque/Labra/2015,Elduque/Labra/2019,SPP-last} for another look in the interplay between evolution algebras and graphs. As we shall see later, there are two natural ways of associating an evolution algebra to a given graph. One of them is obtained whether the structure matrix is the adjacency matrix of the graph. The second one is obtained if we consider as the structure matrix the transition matrix of the symmetric random walk on the same graph. One of the research problems suggested by \cite{tian}, and reinforced later by \cite{tian2}, was to understand the connection between both algebras given a fixed graph. The first results related to this problem appeared in \cite{PMP,PMP1}, where the authors shown for a non-singular graph that both evolution algebras are isomorphic if, and only if, the graph is regular or biregular. Moreover, through examples and partial results they conjectured that the same result remains true for singular graphs. The purpose of this work is to provide new non-trivial examples supporting this conjecture. The paper is organized as follows. Section 2 is devoted to the preliminaries of our work. The main notations, definitions, and known results are stated there. In Section 3 we present three examples of non-regular nor biregular graphs satisfying the conjecture of \cite{PMP2}. More precisely, we show for a family of caterpillar trees, a family of tadpole graphs and a twin-free graph, that the respective evolution algebras are not isomorphic between them.

\section{Evolution algebras associated to a finite graph}
\subsection{Notation of Graph Theory}
We start with some basic notation of Graph Theory. The reader who is familiar with graphs can skip this subsection during the reading of the paper. 

\subsubsection{Basic notation} A finite graph $G$ with $n$ vertices is a pair $(V,E)$ where $V:=\{1,\ldots,n\}$ is the set of vertices and $E\subseteq V \times V$ is the set of edges. If $(i,j)\in E$ we say that $i$ and $j$ are neighbors; we denote the set of neighbors of vertex $i$ by $\mathcal{N}(i)$ and the cardinality of this set by $\deg(i)$. A leaf vertex (also pendant vertex) is a vertex with degree one. If $U\subseteq V$ we use the notation $U^c:=V\setminus U$ and $\mathcal{N}(U)=\{v\in V: v\in \mathcal{N}(u)  \text{ for some } u\in U\}$. The adjacency matrix of a graph $G$ with $n$ vertices is an $n\times n$ symmetric matrix denoted by $A:=A(G)=(a_{ij})$ such that $a_{ij}=1$ if $i$ and $j$ are neighbors and $0$, otherwise. We say that a finite graph is singular if its adjacency matrix $A$ is a singular matrix (or equivalently, if $\det A =0$), otherwise we say that the graph is  non-singular. All the graphs we consider in this work are connected; i.e. for any $i,j\in V$ there exists a positive integer $\ell$ and a sequence of vertices $\gamma=(i_0,i_1,i_2,\ldots,i_\ell)$ such that $i_0=i$, $i_\ell=j$ and $(i_k,i_{k+1})\in E$ for all $k\in\{0,1,\ldots,\ell-1\}$. For simplicity, we consider only graphs which are simple; i.e. without multiple edges or loops.

\subsubsection{Twin vertices and twin partition}

We say that two vertices $i$ and $j$ are twins if they have the same set of neighbors; i.e. $\mathcal{N}(i)=\mathcal{N}(j)$. We notice that by defining the relation $\sim_{t}$ on the set of vertices $V$ by $i\sim_{t} j$ whether $i$ and $j$ are twins, then $\sim_{t}$ is an equivalence relation. An equivalence class of the twin relation is called a twin class. In other words, the twin class of a vertex $i$ is the set $\{j\in V:i \sim_{t} j\}$. 

\subsection{The evolution algebras induced by a finite graph}
An evolution algebra associated to a graph $G$ is defined by \cite[Section 6.1]{tian} as follows.

\smallskip
\begin{definition}\label{def:eagraph}
Let $G=(V,E)$ be a graph with $n$ vertices and adjacency matrix $A=(a_{ij})$. The evolution algebra associated to $G$ is the algebra $\A(G)$ with natural basis $S=\{e_1,\ldots,e_n\}$, and relations
\begin{equation*}
e_i \cdot e_i = \sum_{k=1}^n a_{ik} e_k, \text{for  }i \in  \{1,\ldots,n\} \,\, \text{ and } \,\, e_i \cdot e_j =0,\text{ if }i\neq j.
\end{equation*}
\end{definition}

\begin{example}\label{exa:tadpole41}
The $(n,m)$-tadpole graph, denoted by $T_{n,m}$, is the graph obtained by joining a cycle graph $C_n$ to a path graph $P_m$ with a bridge. The graph $T_{4,1}$, and its adjacency matrix, is represented in Fig.\ref{FIG:graph1}. The evolution algebra $\A(T_{4,1})$ has a natural basis $S=\{e_1,e_2,e_3,e_4,e_5\}$, and relations given by: 

$$
\begin{array}{rclrcl}
e_1^2 &=& e_2 + e_4, & e_2^2 &=& e_1 + e_3,  \\[.2cm]
e_3^2 &=& e_2 + e_4, & e_4^2 &=& e_1 + e_3+e_5,  \\[.2cm]
e_5^2 &=& e_4, & e_i \cdot e_j &=&0,\text{ if }i\neq j.
\end{array}
$$

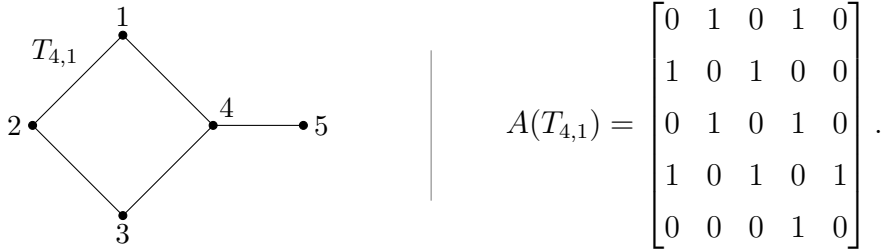
\begin{figure}[h!]
\centering
\begin{tikzpicture}
    \node at (0,0){
\begin{tikzpicture}[scale=0.6]

\draw (0,0) -- (2,0);

\draw (0,0) -- (-2,-2);
\draw (0,0) -- (-2,2);

\draw (-2,-2) -- (-4,0);
\draw (-2,2) -- (-4,0);
\filldraw [black] (-4,0) circle (2.5pt);  \draw (-4.4,0) node {\small$2$};
\filldraw [black] (0,0) circle (2.5pt);  \draw (0.3,0.4) node {\small$4$};
\filldraw [black] (2,0) circle (2.5pt);  \draw (2.4,0) node {\small$5$};
\filldraw [black] (-2,2) circle (2.5pt);  \draw (-2,2.4) node {\small$1$};
\filldraw [black] (-2,-2) circle (2.5pt);  \draw (-2,-2.4) node {\small$3$};
\draw (-3.5,1.6) node {\small $T_{4,1}$};
\end{tikzpicture}};
\node at (7,0){
\begin{tikzpicture}
\node at (0,0) {$A(T_{4,1})=\displaystyle
\begin{bmatrix}
0 & 1 & 0 & 1 & 0\\[.2cm]
1 & 0 & 1 & 0 & 0\\[.2cm]
0 & 1 & 0 & 1 & 0\\[.2cm]
1 & 0 & 1 & 0 & 1\\[.2cm]
0 & 0 & 0 & 1 & 0
\end{bmatrix}.$};
\end{tikzpicture}
};
\draw[gray,thin] (3.5,-1) -- (3.5,1);
\end{tikzpicture}
\caption{The tadpole graph $T_{4,1}$ and its adjacency matrix.}\label{FIG:graph1}
\end{figure}

\end{example}

Another evolution algebra associated to a given graph $G=(V,E)$ is the one induced by the symmetric random walk on $G$, which is a discrete time Markov chain $(X_\ell)_{\ell\geq 0}$ with states space equals to $V$ and transition probabilities given by $\mathbb{P}(X_{\ell+1}=k|X_{\ell}=i)=a_{ik}/\deg(i),$ where $i,k\in V$, $\ell\in \mathbb{N}$, and $\deg(i)=\sum_{k\in V} a_{ik}$ is the number of neighbors of vertex $i$. More precisely, since the random walk on $G$ is a discrete-time Markov chain, we can associate to $G$ its related Markov evolution algebra.

\smallskip
\begin{definition}
Let $G=(V,E)$ be a graph with $n$ vertices and adjacency matrix given by $A=(a_{ij})$. We define the evolution algebra associated to the symmetric random walk on $G$ as the algebra $\A_{RW}(G)$ with natural basis $S=\{e_1,\ldots,e_n\}$, and relations given by
\begin{equation*}
e_i \cdot e_i = \sum_{k=1}^{n}\left( \frac{a_{ik}}{\deg(i)}\right)e_k, \text{ for }i  \in \{1,\ldots,n\}  \text{ and } e_i \cdot e_j =0, \text{ if } i\neq j.
\end{equation*}
\end{definition}

\begin{example}
Consider the Tadpole graph $T_{4,1}$ of Example \ref{exa:tadpole41}. In this case the evolution algebra $\A_{RW}(T_{4,1})$ has natural basis $S=\{e_1,e_2,e_3,e_4,e_5\}$, and relations given by:
$$
\begin{array}{rclrcl}
e_1^2 &=& \frac{1}{2}\left(e_2 + e_4\right), &  e_2^2 &=& \frac{1}{2}\left(e_1 + e_3\right),\\[.2cm]
e_3^2 &=& \frac{1}{2}\left(e_2 + e_4\right), & e_4^2 &=& \frac{1}{3}\left(e_1 + e_3 + e_5\right),\\[.2cm]
e_5^2 &=& e_4, & e_i \cdot e_j &=&0,\text{ if }i\neq j.
\end{array}
$$
\end{example}

\subsection{On the existence of isomorphisms between $\A(G)$ and $\A_{RW}(G)$} 
One question of interest when studying evolution algebras related to graphs has been what is the connection between $\A_{RW}(G)$ and $\A(G)$ for a given graph $G$. This question has been an open problem, stated initially by \cite{tian,tian2}, which has been addressed recently by \cite{PMP,PMP1} mainly for the case of finite graphs. The strategy is the statement of conditions to guarantee the existence of isomorphisms between these algebras.  These works, in turn, motivated the analysis done by \cite{SPP-siberian} for a Hilbert evolution algebra, which is the extension of the concept of evolution algebra proposed by \cite{SPP-indian} in a framework of Hilbert spaces. 

\smallskip

We recall that, given two  $\mathbb{K}$-algebras $\mathcal{A}$ and $\mathcal{B}$, a $\mathbb{K}$-linear map $f:\mathcal{A} \longrightarrow \mathcal{B}$ is an algebra homomorphism if $f(xy)=f(x)f(y)$, for all $x,y \in \mathcal{A}$. If $f$ is a bijection, it is called an isomorphism. In this case, we write $\A\cong \mathcal{B}$.

It has been proved by \cite{PMP} that $\A_{RW}(G) \cong \A(G)$ provided $G$ is well-behaved in some sense. We say that $G=(V,E)$ is a $k$-regular graph if $\deg(i) = k$ for any $i\in V$ and some positive integer $k$. We say that $G$ is a bipartite graph if its vertices can be divided into two disjoint sets, say $V_1$ and $V_2$, such that every edge connects a vertex in $V_1$ to one in $V_2$. Moreover, we say that $G$ is a biregular graph if it is a bipartite graph $G=(V_1,V_2,E)$ for which every two vertices on the same side of the given partition have the same degree as each other. In this case, if the degree of the vertices in $V_1$ is $k_1$ and the degree of the vertices in $V_2$ is $k_2$, then we say that $G$ is a $(k_{1},k_2)$-biregular graph. 

\begin{theorem}\label{theo:criterio} \cite[Theorem 2.3]{PMP1} Let $G$ be a finite non-singular graph. $\A_{RW}(G)\cong \A(G)$ if, and only if, $G$ is a regular or a biregular graph. Moreover, if $\A_{RW}(G)\ncong \A(G)$ then the only homomorphism between them is the null map.
\end{theorem}

In order to prove Theorem \ref{theo:criterio} it has been important to deal with non-singular matrices, see \cite{PMP1} for details. In such work, the authors proved on one hand that $\A_{RW}(G)\cong \A(G)$ provided $G$ is a regular or a biregular graph (see \cite[Theorem 3.2(i)]{PMP} and \cite[Proposition 2.9]{PMP1}), independently of $G$ being a non-singular graph. However, the proof of the reciprocal needs to ask the non-singularity of $G$ (see \cite[Proposition 2.10]{PMP1}). 

Although some partial results are true also for singular matrices, see \cite{PMP,PMP2}, new arguments must be developed to deal with the singular case. It is worth pointing out that recently, \cite[Theorem 2.2.6]{acevedo} pointed out that any non-singular biregular graph must be regular so although Theorem \ref{theo:criterio} should be stated in terms of regular graphs, we expect that for singular graphs regularity or biregularity is a necessary condition for the existence of isomorphisms between the respective evolution algebras.

In order to study this problem in singular graphs we suggest to consider those containing twins and to explore how the twin partition of the graph can help to obtain results. This has been useful to deal with other issues like the characterization of the derivations of an evolution algebra associated to a graph, see \cite{PMP2,reis} for more details. The analysis developed by \cite{PMP,PMP1} motivated the authors to conjecture that, like in the non-singular case, if $\A_{RW}(G)\ncong \A(G)$ then the only evolution homomorphism between them should be the null map. In this work we provide new non-trivial examples supporting this conjecture. The following auxiliary results will be useful for our task. Let us remember that if $G=(V,E)$ is a given graph and $f:\mathcal{A}_{RW}(G) \longrightarrow \mathcal{A}(G)$ is an algebra homomorphism, then the following condition holds for any $i,j\in V$: 

\begin{equation}\label{eq:condf}
f(e_i \cdot e_j) = f(e_i)\cdot f(e_j),\\[.2cm]
\end{equation}

\noindent
where the $e_ i$'s are the generators of the involved natural basis. Our first task will be to translate the previous condition in the language of the neighborhoods of the graph.

\smallskip

\begin{prop}
Let $G=(V,E)$ be a finite graph with adjacency matrix $A(G)=(a_{ij})$ and let $f:\mathcal{A}_{RW}(G) \longrightarrow \mathcal{A}(G)$ be a homomorphism such that 
\begin{equation}\label{eq:homo}
f(e_i) =\sum_{k\in V} t_{ik}e_k,
\end{equation}
for any $i\in V$. Then,
\begin{equation}\label{eq:propa}
\sum_{k\in \mathcal{N}(r)}t_{ik}t_{jk} =0,\text{ for all }r,i, j\in V\text{ with }i\neq j,
\end{equation}
and
\begin{equation}\label{eq:propb}
\sum_{k\in \mathcal{N}(r)} t_{ik}^2 = \sum_{\ell \in \mathcal{N}(i)}\frac{t_{\ell r}}{\deg(i)},\text{ for all }i,r\in V.
\end{equation}
\end{prop}

\begin{proof}
Consider $i,j\in V$ with $i\neq j$. Since $e_i \cdot e_j =0$ we have from \eqref{eq:condf} that
\begin{equation}\label{eq:eqprop1}
0=f(e_i \cdot e_j )=f(e_i) \cdot f(e_j)=\left(\sum_{k\in V}t_{ik}e_k\right)\left(\sum_{k\in V}t_{jk}e_k\right)=\sum_{k\in V}t_{ik}t_{jk}e_k^2,
\end{equation}
but
\begin{equation}\label{eq:eqprop2}
\sum_{k\in V}t_{ik}t_{jk}e_k^2=\sum_{k\in V}t_{ik}t_{jk}\left\{\sum_{r\in V}a_{kr}e_r\right\}=\sum_{r\in V}\left\{\sum_{k\in V}t_{ik}t_{jk}a_{kr}\right\}e_r.
\end{equation}
By \eqref{eq:eqprop1} and \eqref{eq:eqprop2} we have
$$0=\sum_{r\in V}\left\{\sum_{k\in \mathcal{N}(r)}t_{ik}t_{jk}\right\}e_r,$$
which proves \eqref{eq:propa}. Now, let $i\in V$ and observe that

\begin{equation}\label{eq:prop2}
f(e_i^2)=f\left(\sum_{\ell \in \mathcal{N}(i)} \frac{1}{\deg(i)}e_{\ell}\right)=\sum_{\ell \in \mathcal{N}(i)} \frac{1}{\deg(i)}f(e_{\ell})=\sum_{r\in V}\left\{\sum_{\ell \in \mathcal{N}(i)} \frac{t_{\ell r}}{\deg(i)} \right\}e_r,
\end{equation}
while
\begin{equation}\label{eq:prop3}
f(e_i)^2 = \left(\sum_{r\in V}t_{ir}e_r\right)^2 = \sum_{r\in V}t_{ir}^2 e_r^2 = \sum_{r\in V}t_{ir}^2\left\{ \sum_{k\in \mathcal{N}(r)}e_{k}\right\}=\sum_{r\in V}\left\{\sum_{k\in \mathcal{N}(r)} t_{ik}^2\right\} e_r.
\end{equation}\vspace{0.2cm}

Therefore \eqref{eq:propb} is a consequence of \eqref{eq:condf}, \eqref{eq:prop2} and \eqref{eq:prop3}.
\end{proof}

In what follows we assume that if $f:\mathcal{A}_{RW}(G) \longrightarrow \mathcal{A}(G)$ is a homomorphism then, for any $i\in V$, $f(e_i)$ is given by \eqref{eq:homo}. As an immediate consequence of equations \eqref{eq:propa} and \eqref{eq:propb} we have the following corollaries useful for the analysis of some examples.

\begin{corollary}\label{coro:leaf}
Let $G=(V,E)$ be a finite graph and let $f:\mathcal{A}_{RW}(G) \longrightarrow \mathcal{A}(G)$ be a homomorphism. If $\ell\in V$ is a leaf vertex, with $\mathcal{N}(\ell)=\{k_{\ell}\}$, then $f$ satisfies the following conditions:
\begin{enumerate} [label=(\roman*)]
\item \label{eq:nr1} 
 $\displaystyle t_{i\, k_{\ell}}t_{j\,k_{\ell}}= 0$, for all $i,j\in V$ such that $i\neq j$.  
\item  \label{eq:nr2} $\displaystyle t_{i\,k_{\ell}}\neq 0$, for at most one $i\in V.$ 
\item \label{eq:nr3} If there exists $u\in V\setminus\{\ell\}$ such that $u\sim_{t} \ell$, then $\displaystyle t_{u \,k_{\ell}}=t_{\ell k_{\ell}}=t_{k_{\ell}\ell}=t_{k_{\ell}u} =0$. 
\end{enumerate}
\end{corollary}

\begin{proof}
  \ref{eq:nr1} and \ref{eq:nr2} are straightforward. For \ref{eq:nr3} note that from \eqref{eq:propb}, letting $r=i=\ell$ we get $t_{\ell\, k_{\ell}}^2=t_{k_{\ell}\,\ell}$, and letting $r=\ell$ and $i=u$ we get $t_{u\, k_{\ell}}^2=t_{k_{\ell}\,\ell}$. Therefore  $t_{k_{\ell}\,\ell}^2=(t_{\ell\, k_{\ell}}\,t_{u\, k_{\ell}})^2$.  Since $t_{\ell\, k_{\ell}}\,t_{u\, k_{\ell}}=0$, by \ref{eq:nr1}, it follows that $t_{k_{\ell}\,\ell}=0$. This in turn implies that $t_{u\, k_{\ell}}=t_{\ell\, k_{\ell}}=0$. Similarly, the condition $t_{k_{\ell}\,u} = 0$ can be established.
\end{proof}

\smallskip
\begin{corollary}\label{coro:leaf-twins} Let $G=(V,E)$ be a finite graph and let $f:\mathcal{A}_{RW}(G) \longrightarrow \mathcal{A}(G)$ be a homomorphism. Let $\ell, u \in V$ be leaf vertices, $\ell \not = u$, such  that $u\sim_{t} \ell$  and $\mathcal{N}(\ell)=\{k_{\ell}\}$. If $\omega\in V\setminus\{u,\ell\}$ is a leaf vertex with $\mathcal{N}(\omega)=\{k_{\omega}\}$, then $f$ satisfies the following conditions:

\begin{enumerate} [label=(\roman*)]
\item \label{coro:leaf-twins_1} $\displaystyle t_{u \,k_{\omega}}=t_{\ell \,k_{\omega}}=t_{k_{\ell}\,\omega}=0$. 
\item \label{coro:leaf-twins_2} $\displaystyle t_{k_{\omega}\,\ell}=t_{k_{\omega}\,u}=t_{\omega\,k_{\ell}}^2$. 
\end{enumerate}
\end{corollary}

\begin{proof}
    For \ref{coro:leaf-twins_1}, note that from \eqref{eq:propb}, letting $r=\omega$ and $i=\ell$ yields $t_{\ell\, k_{\omega}}^2=t_{k_{\ell}\,\omega}$, and letting $r=\omega$ and $i=u$ yields $t_{u\, k_{\omega}}^2=t_{k_{\ell}\,\omega}$. Consequently $t_{k_{\ell}\,\omega}^2=(t_{\ell\, k_{\omega}}\,t_{u\, k_{\omega}})^2$. However, since $t_{\ell\, k_{\omega}}\,t_{u\, k_{\omega}}=0$, by Corollary \ref{coro:leaf}\ref{eq:nr1}, we conclude that $t_{k_{\ell}\,\omega}=0$, which in turn implies $t_{u\, k_{\omega}}=t_{\ell\, k_{\omega}}=0$.   The statement  \ref{coro:leaf-twins_2} follows from applying \eqref{eq:propb} with $r=u$ and $i=\omega$ and with $r=\ell$ and $i=\omega$, respectively.
\end{proof}

\smallskip
\begin{corollary}\label{coro:column0}
Let $G=(V,E)$ be a finite graph and let $f:\mathcal{A}_{RW}(G) \longrightarrow \mathcal{A}(G)$ be a homomorphism. If there exists $k\in V$ such that $t_{ik}=0$ for any $i\in V$, then $f$ is the null map.
\end{corollary}

\begin{proof}
Let $V_0:=\{ k\in V:t_{ik}=0 \text{ for any }i\in V\}$. By hypothesis $\emptyset \neq V_0 \subseteq V$. Let $V_1:=\mathcal{N}(V_0)\cap V_0^c$. We claim that $t_{ik}=0$ for $i\in V$ and $k\in V_1$. Indeed, for  $k\in V_1$, there is $r\in V_0$ such that $k\in \mathcal{N}(r)$. Since $t_{ir}=0$ for $i\in V$, \eqref{eq:propb} implies $t_{ik}^2=0$ for $i\in V$. In general, define for $n\in \mathbb{N}$ the subset of vertices
$$V_{n}:=\mathcal{N}(V_{n-1})\bigcap \left(\bigcup_{i=0}^{n-1} V_i\right)^c.$$
As $G$ is a finite and connected graph there exists $m\in \mathbb{N}$ such that
$$V=\bigcup_{n=0}^{m}V_n.$$
Moreover, as before, by suitable applications of \eqref{eq:propb} we can conclude that $t_{ik}=0$ for any $i\in V$ and any $k\in V_n$ for $n\in\{0,\ldots,m\}$. Therefore $f=0$.
\end{proof}

\section{The case of some singular graphs} 

\subsection{Caterpillar trees}

The caterpillar trees are trees for which all the vertices are within distance $1$ of a central path. For the sake of simplicity in the exposition we adopt the notation $C_{a_1,a_2,\ldots, a_{n}}$ for a caterpillar tree with a central path of $n$ vertices, and such that the $ith$-vertex of the path is connected with $a_i$ vertices outside the path, for $i\in \{1,\ldots,n\}$. In other words, if $i\in\{2,\ldots,n-1\}$ then the $ith$-vertex of the path has degree $a_{i} + 2$ while the first and the last vertices of the path have degree equal to $a_i+1$ each. See Fig.\ref{FIG:caterpillar} for an illustration of a caterpillar tree.

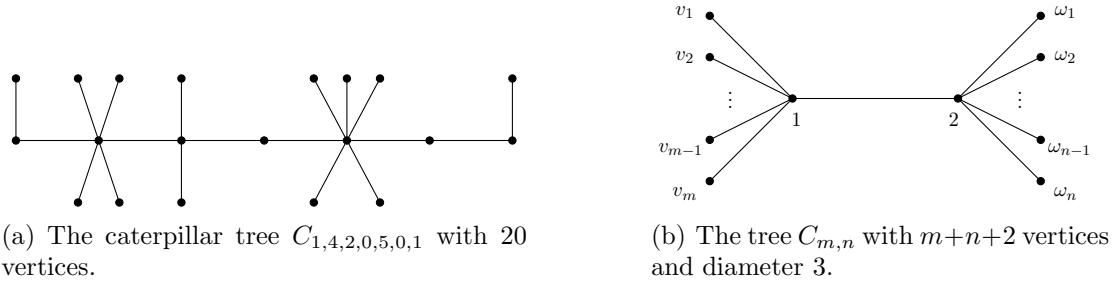
\begin{figure}[!h]
\begin{center}
\subfigure[The caterpillar tree $C_{1,4,2,0,5,0,1}$ with $20$ vertices.]{
\begin{tikzpicture}[scale=0.55, every node/.style={scale=0.55}]

\draw (-6,0) -- (6,0);
\draw (-4,0) -- (-3.5,1.5);
\draw (-4,0) -- (-3.5,-1.5);
\draw (-4,0) -- (-4.5,1.5);
\draw (-4,0) -- (-4.5,-1.5);

\draw (-2,1.5) -- (-2,-1.5);

\draw (2.8,1.5) -- (1.2,-1.5);
\draw (2.8,-1.5) -- (1.2,1.5);
\draw (2,0) -- (2,1.5);


\draw (6,0) -- (6,1.5);
\draw (-6,0) -- (-6,1.5);


\filldraw [black] (-6,1.5) circle (2.5pt);

\filldraw [black] (0,0) circle (2.5pt);
\filldraw [black] (2,0) circle (2.5pt);
\filldraw [black] (4,0) circle (2.5pt);
\filldraw [black] (6,0) circle (2.5pt);
\filldraw [black] (-2,0) circle (2.5pt);
\filldraw [black] (-4,0) circle (2.5pt);
\filldraw [black] (-6,0) circle (2.5pt);

\filldraw [black] (-3.5,1.5) circle (2.5pt);
\filldraw [black] (-4.5,-1.5) circle (2.5pt);
\filldraw [black] (-3.5,-1.5) circle (2.5pt);
\filldraw [black] (-4.5,1.5) circle (2.5pt);

\filldraw [black] (-2,1.5) circle (2.5pt);
\filldraw [black] (-2,-1.5) circle (2.5pt);

\filldraw [black] (2,1.5) circle (2.5pt);
\filldraw [black] (2.8,1.5) circle (2.5pt);
\filldraw [black] (1.2,1.5) circle (2.5pt);
\filldraw [black] (2.8,-1.5) circle (2.5pt);
\filldraw [black] (1.2,-1.5) circle (2.5pt);


\filldraw [black] (6,1.5) circle (2.5pt);

\end{tikzpicture}
\label{FIG:caterpillar(a)}}\qquad \qquad\subfigure[The tree $C_{m,n}$ with $m+n+2$ vertices and diameter $3$.]{\begin{tikzpicture}[scale=0.55, every node/.style={scale=0.7}]

\draw (-2,0) -- (2,0);
\draw (-2,0) -- (-4,2);
\draw (-2,0) -- (-4,1);
\draw (-2,0) -- (-4,-2);
\draw (-2,0) -- (-4,-1);

\draw (2,0) -- (4,2);
\draw (2,0) -- (4,1);
\draw (2,0) -- (4,-2);
\draw (2,0) -- (4,-1);

\filldraw [black] (-2,0) circle (2.5pt);
\filldraw [black] (2,0) circle (2.5pt);
\filldraw [black] (-4,2) circle (2.5pt);
\filldraw [black] (-4,1) circle (2.5pt);
\filldraw [black] (-4,-1) circle (2.5pt);
\filldraw [black] (-4,-2) circle (2.5pt);

\filldraw [black] (4,2) circle (2.5pt);
\filldraw [black] (4,1) circle (2.5pt);
\filldraw [black] (4,-1) circle (2.5pt);
\filldraw [black] (4,-2) circle (2.5pt);

\draw (-1.9,-0.5) node {\small$1$};
\draw (1.9,-0.5) node {\small$2$};

\draw (-4.6,2.1) node {\small$v_{1}$};
\draw (-4.6,1) node {\small$v_{2}$};
\draw (-3.5,0.1) node {\small$\vdots$};
\draw (-4.7,-1.25) node {\small$v_{m-1}$};
\draw (-4.6,-2.3) node {\small$v_{m}$};

\draw (4.6,2.1) node {\small$\omega_{1}$};
\draw (4.6,1) node {\small$\omega_{2}$};
\draw (3.5,0.1) node {\small$\vdots$};
\draw (4.7,-1.25) node {\small$\omega_{n-1}$};
\draw (4.6,-2.3) node {\small$\omega_{n}$};

\end{tikzpicture}
\label{FIG:caterpillar(b)}}
\caption{Caterpillar trees.
\ref{FIG:caterpillar(a)} A caterpillar tree obtained from a path with $7$ vertices. \ref{FIG:caterpillar(b)} A tree with diameter $3$ can be seen as a caterpillar tree. Its set of vertices can be partitioned as $V=V_p \cup V_1 \cup V_2$ with $V_{p}:=\{1,2\}$ (the vertices of a path), $V_1:=\{v_{1},\ldots,v_{m}\}$ is the subset of those vertices outside the path and connected to vertex $1$, and $V_2:=\{\omega_{1},\ldots,\omega_{n}\}$ is the subset of those vertices outside the path and connected to vertex $2$.}\label{FIG:caterpillar}
\end{center}
\end{figure}
 
It is not difficult to see that this is an example of singular path provided $a_i>1$ for some $i\in\{1,\ldots,n\}$. In particular, $C_{m,n}$ is a tree with $m+n+2$ vertices and diameter $3$, see Fig. \ref{FIG:caterpillar(b)}. This case was considered by \cite[Example 2.7]{PMP1}, where the authors showed that the only homomorphism between $\A(C_{2,2})$ and $\A_{RW}(C_{2,2})$ is the null map. In what follows we prove the general result.

\vspace{0.2cm}
\begin{prop}\label{prop:treediam3}Let $m,n\in \mathbb{N}\setminus \{1\}$ and let $C_{m,n}$ be a tree with $m+n+2$ vertices and diameter $3$. Then $\A(C_{m,n}) \ncong \A_{RW}(C_{m,n})$. Moreover, the only homomorphism between $\A(C_{m,n})$ and $\A_{RW}(C_{m,n})$ is the null map.
\end{prop}

\begin{proof}
    We shall prove by contradiction that the only homomorphism between $\A(C_{m,n})$ and $\A_{RW}(C_{m,n})$ is the null map. In order to do it, we think $C_{m,n}$ as a path with the two vertices $V_{p}:=\{1,2\}$, those vertices outside the path and connected to vertex $1$ as $V_1:=\{v_{1},\ldots,v_{m}\}$, and those vertices outside the path and connected to vertex $2$ as $V_2:=\{\omega_{1},\ldots,\omega_{n}\}$. That is, the set of vertices is partitioned as $V=V_p \cup V_1 \cup V_2$. Now, assume that $f:\A_{RW}(C_{m,n})\longrightarrow \A(C_{m,n})$ is a homomorphism. By Corollary \ref{coro:leaf}\ref{eq:nr1} we have that  
    \begin{equation}\label{eq:Cmn1}
        t_{i1}t_{j1}=t_{i2}t_{j2}=0, \text{ for }i,j\in V\text{ with }i\neq j,
    \end{equation}
and by Corollary \ref{coro:leaf}\ref{eq:nr3} we have 
\begin{equation}\label{eq:Cmn2}
        t_{1 v}=t_{v 1}=t_{2 \omega}=t_{\omega 2}=0, \text{ for }v\in V_1\text{ and }\omega \in V_2.
    \end{equation}
In addition, Corollary \ref{coro:leaf-twins}\ref{coro:leaf-twins_1} implies 
\begin{equation}\label{eq:Cmn3}
        t_{1 \omega}=t_{\omega 1}=t_{2 v }=t_{v 2}=0, \text{ for }v\in V_1\text{ and }\omega \in V_2.
    \end{equation}
Then, \eqref{eq:Cmn2} and \eqref{eq:Cmn3} can be read as
\begin{equation}\label{eq:Cmn4}
        t_{i v}=t_{v i}=0, \text{ for }i\in V_p\text{ and }v \in V_1 \cup V_2.
\end{equation}
By \eqref{eq:propa}, with $r\in V_p$ fixed, and by \eqref{eq:Cmn1}, yields
\begin{equation*}
 \sum_{v\in V_r} t_{i v} t_{j v}=0, \text{ for }i,j\in V\text{ with }i\neq j.
\end{equation*}
In particular, by \eqref{eq:Cmn4} the previous equation is reduced to
\begin{equation}\label{eq:Cmn5}
 \sum_{v\in V_r} t_{u v} t_{\omega v}=0, \text{ for }u,v,\omega\in V_1\cup V_2\text{ with }u\neq \omega.
    \end{equation}

Let $\gamma:V_p\rightarrow V_p$ a function defined by $\gamma(1)=2$ and $\gamma(2)=1$. By \eqref{eq:propb}, with $r,i\in V_p$ we obtain

\begin{equation*}\label{eq:Cmn6}
  t_{i \gamma(r)}^2 + \sum_{v\in V_r} t_{iv}^2 = \frac{1}{\deg(i)}\left(t_{\gamma(i)r}+\sum_{\omega \in V_i}t_{\omega r}\right),  
\end{equation*}
which, in turn, implies by \eqref{eq:Cmn4} that

\begin{equation}\label{eq:Cmn7}
 \deg(i) t_{i \gamma(r)}^2  = t_{\gamma(i)r}, \,\, \text{ for }i,r\in V_p.   
\end{equation}
Now, by letting $i\in V_p$ and $v\in V_k$ in \eqref{eq:propb} we have
$$t_{i k}^2 = \frac{1}{\deg(i)}\left(t_{\gamma(i)v}+\sum_{\omega \in V_i}t_{\omega v}\right).$$
But $t_{\gamma(i)v}=0$ if $i\in V_p$ and $v\in V_1\cup V_2$, by \eqref{eq:Cmn4}. Then, we get: 
\begin{equation}\label{eq:Cmn8}
 \deg(i) t_{i k}^2  = \sum_{\omega \in V_i}t_{\omega v}, \,\, \text{ for }i,k\in V_p\text{ and }v\in V_k.   
\end{equation}

\noindent
By \eqref{eq:Cmn8} we have:

$$\displaystyle \deg(i)^2 t_{i k}^4 =  \left(\sum_{\omega \in V_i}t_{\omega v}\right)\left(\sum_{u \in V_i}t_{u v}\right)= \sum_{\omega \in V_i}t_{\omega v}^2 + \sum_{\omega, u \in V_i \atop \omega \neq u}t_{\omega v}t_{u v}
$$

\noindent
and by summing over $V_k$, we obtain

$$\displaystyle (\deg(k)-1) \deg(i)^2 t_{i k}^4 = \sum_{v\in V_k} \sum_{\omega \in V_i}t_{\omega v}^2 + \sum_{v\in V_k}  \sum_{\omega, u \in V_i \atop \omega \neq u}t_{\omega v}t_{u v}.
$$

\noindent
However, \eqref{eq:Cmn5} implies that the second term of the right side in the equality vanishes. Thus, we get 
\begin{equation}\label{eq:Cmn9}
  \displaystyle (\deg(k)-1)\, \deg(i)^2\, t_{i k}^4 = \sum_{v\in V_k} \sum_{\omega \in V_i}t_{\omega v}^2, \,\, \text{ for }i,k\in V_p.   
\end{equation}

\noindent
By letting $r\in V_p$ and $i=v\in V_k$ for $k\in V_p$ in \eqref{eq:propb} we have $t_{v \gamma(r)}^2 + \sum_{u\in V_r} t_{vu}^2= t_{kr}.$ Since $t_{v\gamma(r)}=0$ if $v\in V_k$ and $r\in V_p$, by \eqref{eq:Cmn4}, we obtain: 
\begin{equation}\label{eq:Cmn10}
 \sum_{u\in V_r} t_{vu}^2= t_{kr}, \,\, \text{ for }r,k\in V_p\text{ and }v\in V_k.   
\end{equation}

\noindent
By summing both sides of \eqref{eq:Cmn10} over $V_k$, we obtain

\begin{equation*}
 (\deg(k)-1) t_{kr} = \sum_{v\in V_k}\sum_{u\in V_r} t_{vu}^2, \,\, \text{ for }r,k\in V_p.
\end{equation*}

\noindent
This, in conjunction to \eqref{eq:Cmn9}, yields

\begin{equation}\label{eq:Cmn11}
\displaystyle (\deg(k)-1)\, \deg(i)^2\, t_{i k}^4 = (\deg(i)-1) t_{ik}, \,\, \text{ for }i,k\in V_p.   
\end{equation}

\smallskip
\noindent
Note that $i=1, r=2$ and $i=2,r=1$ in \eqref{eq:Cmn7} implies, respectively, that $(m+1)t_{11}^2=t_{22}$ and $(n+1)t_{22}^2=t_{11}$. This implies

\begin{equation}\label{eq:Cmn12}
\displaystyle t_{11}=(n+1)(m+1)^2t_{11}^4.   
\end{equation}

\noindent
By considering $i=k=1$ in \eqref{eq:Cmn11} we obtain $(m+1)^2t_{11}^4=t_{11}$. Then, by \eqref{eq:Cmn12}, and assuming $t_{11}\neq 0$, we obtain $n=0$, which is a contradiction because we are considering $n\geq 2$. Therefore
\begin{equation}\label{eq:Cmn13}
t_{11}=0.    
\end{equation}

\noindent
By letting $i=r=1$ and $i=r=2$ in \eqref{eq:Cmn7} we obtain, respectively, $(m+1)t_{12}^2=t_{21}$ and $(n+1)t_{21}^2=t_{12}$. This implies

\begin{equation}\label{eq:Cmn14}
\displaystyle t_{21}=(m+1)(n+1)^2t_{21}^4.   
\end{equation}

\noindent
On the other hand, $i=2,k=1$ in \eqref{eq:Cmn11} implies $t_{21}=(m/n)(n+1)^2t_{21}^4$. Then, by \eqref{eq:Cmn14}, and assuming $t_{21}\neq 0$, we obtain $n(m+1)=m$. That is $m(n-1)+n=0$, which is a contradiction because we are considering $m,n\geq 2$. Therefore
\begin{equation}\label{eq:Cmn15}
t_{21}=0.    
\end{equation}

\noindent
The proof is finished by \eqref{eq:Cmn2}, \eqref{eq:Cmn3}, \eqref{eq:Cmn13}, \eqref{eq:Cmn15} and Corollary \ref{coro:column0}.

\end{proof}

Now, let us consider a general caterpillar tree. 

\vspace{0.2cm}
\begin{prop}\label{prop:caterpillarV1}Let $n\in\mathbb{N}\setminus\{1,2\}$ and let $C_a:=C_{1,a_2,a_3,\ldots, a_{n}}$  be a caterpillar tree with $n +\sum_{k=2}^{n} a_k$ vertices and $a_k> 1$ for $k \in \{2, \ldots, n\}$. Then $\A(C_a) \ncong \A_{RW}(C_a)$. Moreover, the only homomorphism between $\A(C_a)$ and $\A_{RW}(C_a)$ is the null map.
\end{prop}

\begin{proof}
We label the vertices of the path as $V_{p}:=\{1,\ldots,n\}$, the vertex connected to $1$ as $v_{1}^1$ and those vertices outside the path and connected to vertex $j$ as $\{v_{1}^j,\ldots,v_{a_j}^{j}\}$, for $j\in \{2,\ldots,n\}$. Let  $f:\A_{RW}(C_a)\longrightarrow \A(C_a)$ be a homomorphism. By Corollary \ref{coro:leaf}\ref{eq:nr1}-\ref{eq:nr2}, we have that  

 \begin{equation}\label{eq:C11}
t_{i\,k}t_{j\,k}= 0, \, \,  \textrm{for} \, \, k \in V_{p},
\end{equation} 
\begin{equation}\label{eq:m1c}
t_{ik}\neq 0, \, \text{ for at most one } \, i\in V, \text{ with} \, k \in V_{p}. 
\end{equation} 

\noindent
By \eqref{eq:propa}, for $r=1$ and $i \neq j$, we obtain $\displaystyle t_{i\, v_{1}^1}t_{j\, v_{1}^1} + t_{i\, 2}t_{j\, 2} = 0,$ and by  \eqref{eq:C11}, for $k=2$, we get $\displaystyle t_{iv_{1}^1}t_{jv_{1}^1} =0$,  for $i \neq j$ so 

\begin{equation}\label{eq:m1v1}
t_{iv_{1}^1}\neq 0 \text{ for at most one }i\in V.
\end{equation} 

\noindent By Corollary \ref{coro:leaf-twins}\ref{coro:leaf-twins_1}, by letting $\omega=v_{1}^1$, we obtain 

\begin{equation}\label{eq:C13}
t_{v_{q}^j 1}=0 \, \, \textrm{and} \, \, t_{jv_{1}^1}=0, \, \, \textrm{for} \, \,j \in V_{p} \setminus \{1\}  \, \, \textrm{and} \, \, q \in \{1, \ldots, a_{j}\}.\vspace{0.2cm}
\end{equation} 

\noindent
On the other hand, Corollary \ref{coro:leaf-twins}\ref{coro:leaf-twins_2} implies for $i,j\in V_{p}\setminus \{1\}$ and $i \neq j$ that, if $q\in\{1,\ldots,a_j\}$, then $ \displaystyle t_{i\, v_{q}^{j}}=t^2_{v_{1}^i \,j}=t^2_{v_{2}^i j}= \cdots = t^2_{v_{a_{i}}^i j}$ while, if $s\in\{1,\ldots,a_i\}$, then $\displaystyle t^2_{v_{s}^{i}\, j}=t_{i\, v_{1}^{j}}=t_{i\, v_{2}^{j}}= \cdots = t_{i\, v_{a_{j}}^{j}}.$ Thus, \eqref{eq:m1c} implies
\begin{equation}\label{eq:C16}
t_{iv_{q}^j}=0 \, \, \textrm{and} \, \, t_{v_{q}^j i}=0, \, \, \textrm{for} \, \,i,j \in V_{p} \setminus \{1\}
\end{equation}  
so  \eqref{eq:C13} and  \eqref{eq:C16} can be read as
\begin{equation}\label{eq:C17}
t_{v_{q}^j i}=0, \, \, \textrm{for} \, \,i \in V_{p}, j \in V_{p} \setminus \{1\}
\end{equation}
and 
\begin{equation}\label{eq:C18}
t_{i v_{q}^j }=0, \, \, \textrm{for} \, \,i \in V_{p} \setminus \{1\}, j \in V_{p},
\end{equation}
respectively. By \eqref{eq:propb}, for $i=v_{q}^j$, with $j \in V_{p} \setminus \{1\}$, and $r=1$, we have $t^2_{v_{q}^j v_{1}^1}+t^2_{v_{q}^j  2}=t_{j 1}$ for $q \in \{1, \ldots, a_{j}\}$. Note that \eqref{eq:C17} implies $t^2_{v_{q}^j  2}=0$ so we obtain $t^2_{v_{1}^j \, v_{1}^1}= \cdots= t^2_{v_{a_{j}}^j \,v_{1}^1}=t_{j\, 1}$ and by \eqref{eq:m1v1} we conclude
\begin{equation}\label{eq:C21}
t_{v_{q}^j v_{1}^1}= 0 \, \, \textrm{and} \, \,   t_{j 1}=0 \, \, \textrm{for} \, \, j \in V_{p} \setminus \{1\}.
\end{equation} 

\noindent
For $i=r=v_{1}^1$ in \eqref{eq:propb}, $t_{v_{1}^1 1}^2= t_{1 v_{1}^1}$ and  letting $i=2$ and $r=v_{1}^1$ in  \eqref{eq:propb} and using  \eqref{eq:C13} and \eqref{eq:C21}, we obtain $t_{1 v_{1}^1}=0$ and therefore $t_{ v_{1}^1 1 }=0$.  So, by the above, \eqref{eq:C13} and \eqref{eq:C21}, we claim that in columns $1$ and $v_{1}^ 1$ all entries, except $t_{1 1}$ and $t_{v_{1}^1 v_{1}^1 }$, are null, namely:

\begin{equation}\label{eq:C21.1}
 t_{i 1}=0, \, \, \textrm{for} \, \, i \in V \setminus \{1\} \hspace{0.5cm} \textrm{and} \hspace{0.5cm}
t_{j v_{1}^1}= 0,   \, \, \textrm{for} \, \, j \in V \setminus \{v_{1}^1 \}.
\end{equation} 

Note that if $t_{11}=0$, then Corollary \ref{coro:column0} implies that $f=0$. Let us assume that $t_{11}\neq 0$. By letting $i=2$ and $r=1$ in \eqref{eq:propb}, we obtain by  \eqref{eq:C21.1}, $t_{22}^2= (1/\deg(2))t_{11} \neq 0$ and by
\eqref{eq:m1c} we obtain 
\begin{equation}\label{eq:C22}
t_{i2}= 0, \, \, \textrm{for} \, \, i \in V \setminus \{2\}.
\end{equation}  

\noindent
For $i=v_{1}^1$ and $r=1$ in \eqref{eq:propb} and using \eqref{eq:C22}, we obtain that, $t_{v_{1}^1 \, v_{1}^1}^2= t_{11}$, and for $i=1$ and $r=v_{1}^1$ in \eqref{eq:propb} and using \eqref{eq:C21.1}, we obtain that, 
$t_{v_{1}^1 \, v_{1}^1}= 2t_{11}^2.$ Therefore we conclude that,
\begin{equation}\label{eq:C22.1}
t_{11}= 2^{-2/3}.
\end{equation}  
For $i=v_{1}^1$ and $r=v_{q}^2$ for  $q \in \{1, \ldots, a_{2}\}$ in \eqref{eq:propb}, and using \eqref{eq:C22} we obtain that,  

\begin{equation}\label{eq:C23}
t_{1 \, v_{q}^2}= 0, \, \, \textrm{for} \, \, q \in \{1, \ldots, a_{2}\}.
\end{equation}  

\noindent
Letting $i=3$ and $r=2$ in \eqref{eq:propb}, we obtain, 
$$t_{31}^2+t_{33}^2+\sum_{q=1}^{a_{2}}t_{3 \, v_{q_{1}}^2}^2 = \frac{1}{\deg(3)}\left(t_{22}+t_{42}+\sum_{q=1}^{a_3}t_{v_{q}^3 \, 2} \right)$$
and for  \eqref{eq:C16}, \eqref{eq:C21.1}-\eqref{eq:C22}, 
$$t_{33}^2=\frac{1}{\deg(3)}t_{22} \neq 0$$
therefore by \eqref{eq:m1c} we obtain, 
\begin{equation}\label{eq:C24}
t_{i3}= 0, \, \, \textrm{for} \, \, i \in V \setminus\{3\}.
\end{equation}  

\noindent
Now for $i=1$ and $r=2$ in \eqref{eq:propb} we obtain, 
$$t_{11}^2+t_{13}^2+\sum_{q=1}^{a_2}t_{1 \, v_{q}^2}^2 = \frac{1}{2}(t_{v_{1}^1 2} + t_{22} )$$
and for \eqref{eq:C22}, and \eqref{eq:C23}-\eqref{eq:C24}, $t_{11}^2=\frac{1}{2}t_{22}.$
So, by \eqref{eq:C22.1}, $ t_{22}=2 t_{11}^2=2(2^{-2/3})^2=2^{- 1/3}$ and since $t_{11}=\deg(2) t_{22}^2$, then $\deg(2)=t_{11}/t_{22}^2=2^{-2/3}/2^{-2/3}=1,$ which is a contradiction because $\deg(2)=a_2+2 > 2$. Therefore $t_{11}=0$ and $f=0$. 

\end{proof}


\subsection{Tadpole graphs} 

We consider the $(4,m)$-tadpole graph, with $m\geq 1$, and $m$ odd. See Fig.\ref{FIG:tadpole}.  It should be noted that the $(n,m)$-tadpole graph, with $n\geq 3$, consists of a cycle graph on $n$ vertices and a path graph on $m$ vertices, connected with a bridge. The graph considered in Example \ref{exa:tadpole41} is a $(4,1)$-tadpole graph. We adopt the standard notation $T_{n,m}$, and as it has $n+m$ vertices, we label them as shown in Fig.\ref{FIG:tadpole}. The evolution algebra $\A(T_{4,m})$ has a natural basis $S=\{e_1,\ldots,e_{4+m}\}$,  and the relations between the elements of this basis are given by

\begin{equation*}
\begin{array}{ll}
e_i^2 = e_{i-1} + e_{i+1},&\text{for  }i \notin  \{1,4,4+m\},\\[.3cm]
e_{1}^2 = e_{2} + e_{4},&\\[.3cm]
e_{4}^2 = e_{1} + e_{3} + e_{5},&\\[.3cm]
e_{4+m}^2 = e_{3+m} , &\\[.3cm]
e_i \cdot e_j =0,&\text{if }i\neq j.
\end{array}
\end{equation*}
\noindent
On the other hand, the evolution algebra $\A_{RW}(T_{4,m})$ has a natural basis $S=\{e_1,\ldots,e_{4+m}\}$, and relations given by
\begin{equation*}
\begin{array}{ll}

e_i^2 = \frac{1}{2}(e_{i-1} + e_{i+1}),&\text{for  }i  \notin  \{1,4,4+m\},\\[.3cm]
e_{1}^2 = \frac{1}{2}(e_{2} + e_{4}),&\\[.3cm]
e_{4}^2 =\frac{1}{3}(e_{1} + e_{3} + e_{5}),&\\[.3cm]
e_{4+m}^2 = e_{3+m} ,&\\[.3cm]
e_i \cdot e_j =0,&\text{if }i\neq j.

\end{array}
\end{equation*}

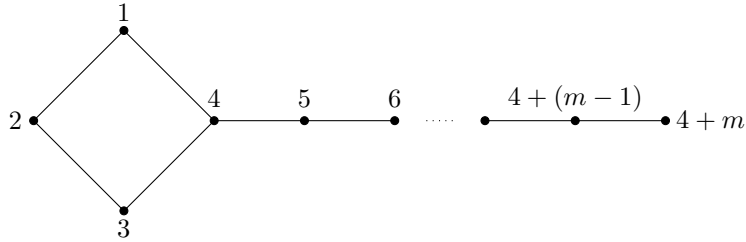
\begin{figure}[!h]
\centering
\begin{tikzpicture}[scale=0.6]

\draw (0,0) -- (2,0) -- (4,0);
\draw (6,0) -- (8,0);
\draw (8,0) -- (10,0);
\draw[dotted] (4.7,0) -- (5.4,0);

\draw (0,0) -- (-2,-2);
\draw (0,0) -- (-2,2);

\draw (-2,-2) -- (-4,0);
\draw (-2,2) -- (-4,0);
\filldraw [black] (-4,0) circle (2.5pt);  \draw (-4.4,0) node {\footnotesize$2$};
\filldraw [black] (0,0) circle (2.5pt);  \draw (0,0.5) node {\footnotesize$4$};
\filldraw [black] (8,0) circle (2.5pt);  \draw (8,0.5) node {\footnotesize$4+(m-1)$};
\filldraw [black] (6,0) circle (2.5pt);
\filldraw [black] (4,0) circle (2.5pt);  \draw (4,0.5) node {\footnotesize$6$};
\filldraw [black] (2,0) circle (2.5pt);  \draw (2,0.5) node {\footnotesize$5$};
\filldraw [black] (10,0) circle (2.5pt);  \draw (11,0) node {\footnotesize$4+m$};
\filldraw [black] (-2,2) circle (2.5pt);  \draw (-2,2.4) node {\footnotesize$1$};
\filldraw [black] (-2,-2) circle (2.5pt);  \draw (-2,-2.4) node {\footnotesize$3$};
\end{tikzpicture}
\caption{$(4,m)$-tadpole graph $T_{4,m}$.} \label{FIG:tadpole}
\end{figure}

\begin{prop}\label{prop:tadpole}
Let $T_{4,m}$ be  the $(4,m)$-tadpole graph with $m\geq 1$ odd. Then the only homomorphism between $\mathcal{A}_{RW}(T_{4,m})$ and $\mathcal{A}(T_{4,m})$ is the null map. Thus $\mathcal{A}_{RW}(T_{4,m})\ncong \mathcal{A}(T_{4,m})$ as evolution algebras.  
\end{prop}

\begin{proof}
Assume that $f:\mathcal{A}_{RW}(T_{4,m}) \longrightarrow \mathcal{A}(T_{4,m})$ is an algebra homomorphism. We  appeal to a suitable application of the Eq.\eqref{eq:propa}-\eqref{eq:propb} to show that all the images of $f$ are zero. Let us start with some equations resulting from \eqref{eq:propa}, for any $i\neq j$:
\begin{eqnarray}
t_{i\, (r-1)}t_{j\, (r-1)} + t_{i\, (r+1)}t_{j\, (r+1)} = 0, & \text{ for }r\notin\{ 1,4,4+m\},\label{eq:l1}\\[.2cm]
t_{i\, 2}t_{j\, 2} + t_{i\, 4}t_{j\, 4} = 0, & \text{ for }r=1,\label{eq:l2}\\[.2cm]
t_{i\, 1}t_{j\, 1} + t_{i\, 3}t_{j\, 3} +   t_{i\, 5}t_{j\, 5} = 0, & \text{ for }r=4,\label{eq:l3}\\[.2cm]
t_{i\, (3+m)}t_{j\, (3+m)} = 0, & \text{ for }r=4+m.\label{eq:l4}
\end{eqnarray}
Additionally, from \eqref{eq:l1}, with $r=2$, and \eqref{eq:l3} we obtain:
\begin{equation}\label{eq:l5}
    t_{i\, 5}t_{j\, 5} =0, \text{ for }i\neq j.
\end{equation}
Which, together with \eqref{eq:l1} for $r\geq 6$, and $r$ even, allow us to conclude that:
$$    t_{i\, \ell}t_{j\, \ell} =0, \text{ for }i\neq j, \text{ and } \ell \geq 5 \text{ with }\ell \text{ odd. }
$$
Moreover, since $m$ is odd then $3+m$ is even so  \eqref{eq:l1} for $r\geq 5$, $r$ odd, \eqref{eq:l2}, and \eqref{eq:l4}  allows to conclude that:
$$    t_{i\, \ell}t_{j\, \ell} =0, \text{ for }i\neq j, \text{ and } \ell \geq 2 \text{ with }\ell \text{ even. }$$
Thus, what we obtained is that
$ t_{i\, \ell}t_{j\, \ell} =0$, for   $i\neq j$, and $\ell \notin \{1,3\},$ which in turn implies:
\begin{equation}\label{eq:max1}
\text{for each }\ell \notin\{1,3\} \text{ there is at most  }i_{\ell}\text{ such that }t_{i_{\ell} \ell}\neq 0.
\end{equation}

Note that, for $\ell \in\{1,3\}$, \eqref{eq:l3} and \eqref{eq:l5} implies:
\begin{equation}\label{eq:crz}
    t_{i\, 1}t_{j\, 1} + t_{i\, 3}t_{j\, 3}=0,\text{ for }i\neq j.
\end{equation}

Since $1\sim_{t} 3$, then by \eqref{eq:propb}, for $r=1$ and for $r=3$, gives us for any $i$: 
\begin{equation}\label{eq:l6}
    \sum_{\ell\in \mathcal{N}(i)} t_{\ell 1} = \sum_{\ell\in \mathcal{N}(i)} t_{\ell 3}.
\end{equation}
Now, using \eqref{eq:propb} with $i=1$ and $i=3$  we have that
\begin{equation}\label{eq:l7}
    \sum_{k\in \mathcal{N}(r)} t_{1k}^2 = \sum_{k\in \mathcal{N}(r)} t_{3k}^2= \frac{1}{2}(t_{2 r}+t_{4 r}) , \,\, \text{ for any } r.
\end{equation}
In  the equation before, if $r\not \in \{ 1,4,4+m\}$, then:
\begin{equation}\label{eq:l8}
    t_{1\, (r-1)}^2 + t_{1\, (r+1)}^2  = t_{3\, (r-1)}^2 + t_{3\, (r+1)}^2. 
\end{equation}
We will prove now four claims.

\vspace{0.2cm}
\noindent
{\bf \underline{Claim 1}:} If $i$ and $j$ have different parity then $t_{ij}=0$.

In fact, by letting $r=4+m$ in \eqref{eq:l7}, we get that $t_{1\, (3+m)}^2 = t_{3\, (3+m)}^2$, that implies  that $t_{1\, (3+m)} \not =0$ if and only if $t_{3\, (3+m)} \not =0$.  Then by \eqref{eq:max1} we can concluded that $t_{1\, (3+m)} = t_{3\, (3+m)}=0$. Now, by letting $r= 2+m$ in \eqref{eq:l7}, we have that $$t_{1\, (1+m)}^2 + t_{1\, (3+m)}^2  = t_{3\, (1+m)}^2 + t_{3\, (3+m)}^2.$$ Then $t_{1\, (1+m)}^2  = t_{3\, (1+m)}^2$ and therefore $t_{1\, (1+m)} = t_{3\, (1+m)}=0$. Repeating the same argument with $r$ odd
we obtain that
\begin{equation} \label{eq:ind1} t_{1j}=t_{3j}=0, \, \textrm{for} \, j \, \textrm{even}.
\end{equation}

The last equation together with \eqref{eq:l7} implies that  $t_{2\,j} +t_{4\,j}=0, \, \textrm{for} \, j \, \textrm{odd}$. Then  by \eqref{eq:max1} we have 
$t_{2\, j}=t_{4\, j}=0,$  for $j$ odd and   $j \not \in\{ 1,3\}.$  Now taking $i=r=2$ in \eqref{eq:propb} we get 
 $$t_{2\, 1}^2 + t_{2\, 3}^2  = \frac{1}{2}(t_{1\, 2} + t_{3\, 2})=0.$$ 
Therefore, using \eqref{eq:max1}, we have  
$t_{2\, 1}=t_{2\, 3}=0$ and by \eqref{eq:l7} with $r=1,3$ we have $t_{4\, 1}=t_{4\, 3}=0$. Thus
\begin{equation}\label{eq:ind2} t_{2\, j}=t_{4\, j}=0, \textrm{ for }   j  \textrm{ odd.} \end{equation}
We note that if $r$ is even then all elements in $\mathcal{N}(r)$ are odd, then by letting $i=4$ in \eqref{eq:propb} and using  \eqref{eq:ind2} we obtain 
\begin{equation}\label{ind3}
    0=\sum_{k\in \mathcal{N}(r)} t_{4k}^2 = \frac{1}{3}(t_{1 r}+t_{3 r} +t_{5 r})= t_{5 r}, \, \, \textrm{ for } \,  r \, \textrm{ even}.
\end{equation}
Now by letting $i=5$ and  $r$ odd  in \eqref{eq:propb} and using  \eqref{eq:ind2} we have that
\begin{equation}\nonumber 
    \sum_{k\in \mathcal{N}(r)} t_{5k}^2 = \frac{1}{2}(t_{4r}+t_{6 r})=\frac{1}{2}t_{6 r}.
\end{equation}
But if $r$ is odd all elements in $\mathcal{N}(r)$ are even, then using \eqref{ind3}  we obtain
\begin{equation}\label{ind4}
 t_{6j}=0  \, \textrm{ for } j  \, \textrm{ odd}.
\end{equation}
Continuing with this process we can conclude that $t_{ij}=0 $ whenever $i$ and $j$ have different parity.

\vspace{0.2cm}

{\bf \noindent \underline{Claim 2}:} If $i\in \{1,2,3, 4\}$ and $j\in \{5,\ldots,4+m\}$, then $t_{ij}=0$.

Using  Claim 1 it remains to prove that $t_{ij}=0$ if $i$ and $j$ have the same parity. From \eqref{eq:l7} with $r=2$ and $r=4$ we obtain that
$$t_{1\, 1}^2 + t_{1\, 3}^2  = t_{3\, 1}^2 + t_{3\, 3}^2\,\, \textrm{ and } \,\, t_{1\, 1}^2 + t_{1\, 3}^2 + t_{1\, 5}^2  = t_{3\, 1}^2 + t_{3\, 3}^2 + t_{3\, 5}^2,$$
respectively. Therefore $t^2_{15}=t^2_{35}$ and then, by \eqref{eq:max1}, we have that $t_{15}=t_{35}=0.$ Now, using \eqref{eq:l8} with $r=6$, we have $t_{1\, 5}^2 + t_{1\, 7}^2  = t_{3\, 5}^2 + t_{3\, 7}^2$, and therefore $t_{1\,7}=t_{3\,7}=0$. We continue using \eqref{eq:l8}, for all possible values of $r$ even until we get that $t_{1\,j}=t_{3\,j}=0$, and  for $j$ odd $$t_{1\, j}=t_{3\, j}=0, \, \textrm{for}\, \, j \neq 1,3.$$
The above in \eqref{eq:l7} implies that $t_{2\, r} + t_{4\, r}=0$ for $r \neq 2,4$ and, by 
 \eqref{eq:max1}, that 
 $$t_{2\, k}=t_{4\, k}=0, \, \textrm{for} \, k \neq 2,4.$$
\vspace{0.2cm}
\noindent 
{\bf \underline{Claim 3}:} If $i\in \{5,\ldots,4+m\}$ and $j>i$, then $t_{ij}=0$.
 
If $i=4$ in \eqref{eq:propb}, then $$\sum_{k\in \mathcal{N}(r)} t_{4k}^2= \frac{1}{3}(t_{1\, r} + t_{3\, r} + t_{5\, r}).$$
We note that if  $r \neq 1,3,5$, then $k \neq 2,4$ and, by Claim 2, $t_{1r}=t_{3r}=t_{4k}=0$. Therefore
\begin{equation}\label{eq:l9}
t_{5\, j}=0, \, \textrm{for} \, j \neq 1,3,5. \end{equation}
Now from \eqref{eq:propb} with $i>4$ we obtain
\begin{equation}\label{eq:l10}
\sum_{k\in \mathcal{N}(r)} t_{ik}^2= \frac{1}{2}(t_{(i-1)\, r} + t_{(i+1)\, r}).
\end{equation}
If $i=5$ and $r>6$  in \eqref{eq:l10} then 
$$0=\sum_{k\in \mathcal{N}(r)} t_{5k}^2= \frac{1}{2}(t_{4\, r} + t_{6\, r})=t_{6\, r}, $$
where we use {\bf Claim 2} and  \eqref{eq:l9}. We continue the process in an analogous way by increasing the value of $i$ in \eqref{eq:l10}. 

\vspace{0.2cm}
{\bf \noindent \underline{Claim 4}:} $3t_{44}^2= t_{55}$ and  $2t_{ii}^2= t_{(i+1) \, (i+1)}$, for $5 \leq i \leq 3+m$.

If $i=4$ and $r=5$ in \eqref{eq:propb}, then $ t_{44}^2 + t_{46}^2= \frac{1}{3}(t_{15} + t_{35}+t_{55} ),$
but $t_{46}=t_{15}=t_{35}=0$ for claim 2, therefore  $3t_{44}^2= t_{55}.$
If $i=3+m$ and $r=4+m$ in \eqref{eq:propb}, then we have $ t_{(3+m) \, (3+m)}^2= \frac{1}{2}(t_{(2+m)\, (4+m)} + t_{(4+m)\, (4+m)} ),$
but $t_{(2+m)\, (4+m)}=0$ for {\bf Claim 3}, therefore  $2t_{(3+m) \, (3+m)}^2= t_{(4+m) \, (4+m)}$. 
For $5 \leq i \leq 2+m$ and $r=i+1$ in \eqref{eq:propb}:
$$ t_{ii}^2 + t_{i \, (i+2)}^2= \frac{1}{2}(t_{(i-1)\, (i+1)} + t_{(i+1)\, (i+1)} ),$$
and since $t_{i\, (i+2)}=t_{(i-1)\, (i+1)}=0$  for Claim 2 and {\bf Claim 3}, then:
$$ 3t_{44}^2= t_{55} \, \, \textrm{and} \, \,  2t_{ii}^2= t_{(i+1) \, (i+1)} \, \,  \textrm{for} \, \,  5 \leq i \leq 3+m.$$ 
\vspace{0.2cm}

\noindent
Using all of the above, we will now show that $f$ is the null homomorphism, for this we will focus on the term $t_{(4+m) \, (4+m)}$. Note that if $t_{(4+m) \, (4+m)}=0$, then by {\bf Claim 2} and {\bf Claim 3},  $t_{i \, (4+m)}=0$ for all $i\in V$, so Corollary \ref{coro:column0} implies that $f$ is the null map.

\vspace{0.2cm}

Let us prove  that $t_{(4+m) \, (4+m)} = 0$. In fact,  if $t_{(4+m) \, (4+m)} \neq 0$, then by {\bf Claim 4}, $t_{ii} \neq 0$ for all $4 \leq i \leq 4+m.$ In particular $t_{44} \neq 0$ and therefore, by \eqref{eq:max1}, $t_{24}=0$. Letting $i=2$ in \eqref{eq:l6} and letting $i=2$ and $r=3$ in \eqref{eq:propb}, we obtain:
\begin{equation}\label{eq:t13} t_{11} +t_{31} = t_{1\, 3} + t_{33} = 2t_{22}^2.
\end{equation}
\noindent
Now letting $r=2$ and $r=4$ in  \eqref{eq:l7} and using  {\bf Claim 2}, we obtain:
\begin{equation}\label{eq:ct13}  \frac{1}{2}(t_{2\,2}+ t_{4\,2}) = t_{1\,1}^2 + t_{13}^2=t_{3\,1}^2 + t_{33}^2= \frac{1}{2}t_{4\,4} \neq 0.
\end{equation}
If $t_{22}=0$, then by \eqref{eq:t13}, $t_{1\,1}+ t_{3\,1}=t_{1\,3}+ t_{3\,3}=0$ and therefore
$$0=(t_{1\,3}+ t_{3\,3})^2 +(t_{1\,1}+ t_{3\,1})^2=t_{1 \, 3}^2+ t_{3\,3}^2 +t_{1\,1}^2+ t_{3\,1}^2 + 2(t_{1\,3} t_{3\,3} +t_{1\,1} t_{3\,1}).$$
We conclude by \eqref{eq:crz} that $t_{1\,3}=t_{3\,3}=t_{1\,1}=t_{3\,1}=0$, what contradicts \eqref{eq:ct13}. 

\noindent
If $t_{22}\neq 0$, then $t_{42}=t_{62}=0$ and by \eqref{eq:ct13}, $t_{22}=t_{44}$. 
Letting $i=5$ and $r=2$ in \eqref{eq:propb} we obtain 
$t_{5\, 1}^2 + t_{53}^2= \frac{1}{2}(t_{4 \, 2}+t_{62})=0$, namely, $t_{51}=t_{53}=0$. Letting $i=4$ and $r=3$ in \eqref{eq:propb} and using \eqref{eq:t13} we obtain that, 
$t_{4\, 4}^2 = \frac{1}{3}(t_{1 \, 3}+ t_{3 \, 3}+t_{5\, 3})=\frac{1}{3}(t_{1 \, 3}+ t_{3 \, 3})=\frac{2}{3}t_{22}^2=\frac{2}{3}t_{44}^2$, what would take us to $t_{44}=0$ and we get a contradiction showing that this case is not possible. 

\end{proof}

\subsection{The bull graph: a twin-free singular graph}
In Propositions \ref{prop:caterpillarV1} and \ref{prop:tadpole} we examine families of singular graphs that contain twins between their vertices. We consider the graph $C_{m,n}$ with $m,n\in \mathbb{N}\setminus\{1\}$, and the graph $C_{1,a_2,a_3,\ldots, a_{n}}$ with $n\in \mathbb{N}\setminus\{1\}$ and $a_i>1$ for $i\in\{2,\ldots,n\}$.  Additionally, we consider the graph $T_{4,m}$ for $m\in \mathbb{N}$ odd. This helps to facilitate our analysis, which aims to prove that the only homomorphism between the respective evolution algebras is the null map. In what follows, we analyse a singular twin-free graph. In the set of twin-free graphs, it has been demonstrated that two different behaviors can be observed in terms of the existence of isomorphisms between the corresponding evolution algebras, provided that the graph is singular.  The cycle graph $C_{n}$ with $n=4k$ for $k \geq 2$ is an example of a singular twin-free graph. Since it is a $2$-regular graph, we know that $\A_{RW}(C_{n})\cong \A(C_{n})$ (see \cite[Theorem 3.2(i)]{PMP} and \cite[Proposition  2.9]{PMP1}). In contrast, the path graph $P_{n}$ with $n \geq 5$ odd represents another example of singular twin-free graph. However,  it has been demonstrated that $\A_{RW}(P_{n}) \not\cong \A(P_{n})$ and the only homomorphism between these algebras is the null map (see \cite[Proposition 3.3]{PMP1}).

\begin{example}
Consider the bull graph $B$, which is a graph with five vertices and five edges. It is represented by a triangle with two disjoint pendant edges. See Fig.\ref{FIG:twin-free-singular}. 

\begin{figure}[!h]
\centering

\begin{tikzpicture}
    \node at (0,0){
\begin{tikzpicture}[scale=0.6]


\draw (0,0) -- (-2,-2.5);
\draw (0,0) -- (1,2.5);
\draw (-4,0) -- (0,0);
\draw (-2,-2.5) -- (-4,0);
\draw (-5,2.5) -- (-4,0);
\filldraw [black] (-4,0) circle (2.5pt);  \draw (-4.4,0) node {\footnotesize$3$};
\filldraw [black] (0,0) circle (2.5pt);  \draw (0.4,0) node {\footnotesize$4$};
\filldraw [black] (-5,2.5) circle (2.5pt);  \draw (-5,3) node {\footnotesize$1$};
\draw (-2,0.7) node {\footnotesize $B$};
\filldraw [black] (1,2.5) circle (2.5pt);  \draw (1,3) node {\footnotesize$2$};
\filldraw [black] (-2,-2.5) circle (2.5pt);  \draw (-2,-3) node {\footnotesize$5$};
\end{tikzpicture}};
\node at (7,0){
\begin{tikzpicture}
\node at (0,0) {$A(B)=\displaystyle
\begin{bmatrix}
0 & 0 & 1 & 0 & 0\\[.2cm]
0 & 0 & 0 & 1 & 0\\[.2cm]
1 & 0 & 0 & 1 & 1\\[.2cm]
0 & 1 & 1 & 0 & 1\\[.2cm]
0 & 0 & 1 & 1 & 0
\end{bmatrix}.$};
\end{tikzpicture}
};
\draw[gray,thin] (3.5,-1) -- (3.5,1);
\end{tikzpicture}
\caption{The bull graph.} \label{FIG:twin-free-singular}
\end{figure}
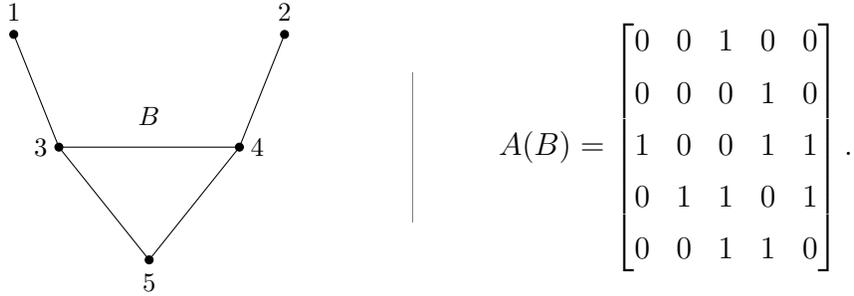

$B$ is a twin-free singular graph and is associated with the following evolution algebras:

\begin{equation*}
\begin{array}{ccc}
\A(B)=\left\{
\begin{array}{l}
e_{1}^2 = e_{3},\\[.3cm]
e_{2}^2 = e_{4},\\[.3cm]
e_{3}^2 = e_{1}+e_{4}+e_{5},\\[.3cm]
e_{4}^2 = e_{2}+e_{3}+e_{5},\\[.3cm]
e_{5}^2 = e_{3}+e_{4}, \\[.3cm]
e_i \cdot e_j =0,\text{ if }i\neq j.
\end{array}\right.
&&
\A_{RW}(B)=\left\{
\begin{array}{l}
e_{1}^2 = e_{3},\\[.3cm]
e_{2}^2 = e_{4},\\[.3cm]
e_{3}^2 = \frac{1}{3}(e_{1}+e_{4}+e_{5}),\\[.3cm]
e_{4}^2 = \frac{1}{3}(e_{2}+e_{3}+e_{5}),\\[.3cm]
e_{5}^2 = \frac{1}{2}(e_{3}+e_{4}), \\[.3cm]
e_i \cdot e_j =0,\text{ if }i\neq j.
\end{array}\right.
\end{array}
\end{equation*}

Clearly, $B$ is a singular graph.  Indeed, the last row of $A(B)$ is the sum of the first and second rows, which is the translation of $e_1^2+e_2^2=e_5^2$. Let us demonstrate that if $f:\A_{RW}(B)\longrightarrow \A(B)$ is a homomorphism, then $f=0$. In fact, is $f$ is a homomorphism then Corollary \ref{coro:leaf}\ref{eq:nr1}-\ref{eq:nr2} implies
\vspace{0.2cm} 
\begin{equation}\label{eq:B1}
t_{i\,k}t_{j\,k}= 0, \, \,  \textrm{for}\,\, i,j\in V\,\,i\neq j, \text{ and }\, \, k \in \{3,4\},
\end{equation}  \vspace{0.2cm} 
and
\begin{equation}\label{eq:B2}
t_{ik}\neq 0, \, \text{ for at most one } \, i\in V, \text{ and } \, k \in \{3,4\}. 
\end{equation}  

\vspace{0.2cm}
\noindent
By \eqref{eq:propb}, for $r=i=1$, we have $t_{13}^2=t_{31}$, and for $r=1$ and $i=2$, we have $t_{23}^2=t_{41}$. If we let $r=1$ and $i=5$ in \eqref{eq:propb} then we get $2t_{53}^2=t_{31}+t_{41}=t_{13}^2 + t_{23}^2$ which in turn, by \eqref{eq:B2}, implies
\begin{equation}\label{eq:B3}
t_{13}=t_{23}=t_{53}=0 \text{ so }t_{31}=t_{41}=0.
\end{equation}

\noindent
In a similar way, we obtain by considering $r=2$ and $i=1,2,5$ the equations $t_{14}^2=t_{32}$, $t_{24}^2=t_{42}$ and $2t_{54}^2=t_{32}+t_{42}=t_{14}^2+t_{24}^2$, respectively. As before, these equations imply, by \eqref{eq:B1} that:
\begin{equation}\label{eq:B4}
    t_{54}=t_{14}=t_{24}=0 \text{ so }t_{32}=t_{42}=0.
\end{equation}

\noindent
By letting $r=5$ and $i=1,2$ in \eqref{eq:propb} we obtain $t_{13}^2+t_{14}^2=t_{35}$ and $t_{23}^2+t_{24}^2=t_{45}$, respectively, which with \eqref{eq:B3}-\eqref{eq:B4} implies:
\begin{equation}\label{eq:B5}
    t_{35}=t_{45}=0.
\end{equation}

\noindent
If we consider $r=3$, $i=4$ and $r=4$, $i=3$ in \eqref{eq:propb}, together with \eqref{eq:B3}-\eqref{eq:B5}, then we can obtain:
\begin{equation}\label{eq:B6}
    3\,t_{44}^2=t_{33},\,\,\text{ and 
  }\,\,3\,t_{33}^2=t_{44}.
\end{equation}

\noindent
We will prove that $t_{33}=t_{43}=0$ which implies, by \eqref{eq:B3} and Corollary \ref{coro:column0} that $f$ is the null map. We prove it by contradiction so let us assume that $t_{33}\neq 0$ or $t_{43}\neq 0$. 

\smallskip
\noindent
{\bf \underline{Case 1}:} $t_{33}\neq 0$. In this case, \eqref{eq:B6} implies that $t_{33} =1/3$. Moreover, $t_{44} \neq 0$ so by \eqref{eq:B1} it should be:
\begin{equation}\label{eq:B7}
    t_{34}=t_{43}=0.
\end{equation}

\noindent
By letting $r=4,$ $i=1$ and $r=3,$ $i=2$  in \eqref{eq:propb} we get, respectively, $t_{12}^2+t_{13}^2+t_{15}^2=t_{34}$ and $t_{21}^2+t_{24}^2+t_{25}^2=t_{43}$, which together with \eqref{eq:B7} implies:
\begin{equation}\label{eq:B8}
    t_{12}=t_{15}=t_{21}=t_{25}=0.
\end{equation}
By letting $r=1,$ $i=4$ in \eqref{eq:propb} we get $3t_{43}^2=t_{21}+t_{31}+t_{51}$, which together with \eqref{eq:B3}, \eqref{eq:B7} and \eqref{eq:B8} implies:
\begin{equation}\label{eq:B9}
    t_{51}=0.
\end{equation}

\noindent
On the other hand, by considering $r=3,$ $i=1$ and $r=1,$ $i=3$ in \eqref{eq:propb} we obtain, respectively, $t_{11}^2+t_{14}^2+t_{15}^2=t_{33}$ and $3t_{33}^2=t_{11}+t_{41}+t_{51}$, which together with \eqref{eq:B3}, \eqref{eq:B4}, \eqref{eq:B8} and \eqref{eq:B9} implies $t_{11}^2=t_{33}$ and $t_{11}=3t_{33}^2$. This in turn implies that $t_{33}^2=1/9$ which is a contradiction because we already showed that it should be $t_{33}=1/3$. Therefore we get,
\begin{equation}\label{eq:B10}
    t_{33}=t_{44}=0.
\end{equation}

\smallskip

{\bf \noindent \underline{Case 2}:} $t_{43}\neq 0$. By letting $r=i=3$ and $r=i=4$ in \eqref{eq:propb} we get, respectively, $3(t_{31}^2+t_{34}^2+t_{35}^2)=t_{13}+t_{43}+t_{53}$ and $3(t_{42}^2+t_{43}^2+t_{45}^2)=t_{24}+t_{34}+t_{54}$, which together with \eqref{eq:B3}-\eqref{eq:B4} imply, respectively, that $3 t_{34}^2=t_{43}$ and $3t_{43}^2=t_{34}$. Thus $t_{34}\neq 0$ and $t_{43}=1/3$. Moreover $t_{34}=1/3$. Now, by the equations resulting from $i=1$, $r=3$ and $i=3$ and $r=1$, together with \eqref{eq:B10} we conclude that
\begin{equation}\label{eq:B11}
    t_{11}=t_{15}=t_{51}=0.
\end{equation}

\noindent
By assuming $r=3$ and $i=5$ in \eqref{eq:propb}, together with \eqref{eq:B4}, \eqref{eq:B10} and \eqref{eq:B11}, we get $2t_{55}^2=t_{43}$. Since $t_{43}=1/3$ it should be $t_{55}=1/\sqrt{6}$. By letting $r=5$ and $i=3$ in \eqref{eq:propb}, together with \eqref{eq:B5}, \eqref{eq:B10} and \eqref{eq:B11}, we get $3t_{34}^2=t_{55}$ but $t_{34}=1/3$ so $t_{55}=1/3$, which is a contradiction. Therefore $t_{43}=0$ and the proof is completed.   
\end{example}

\section*{Acknowledgements}

A portion of this research was conducted during a visit by S. Vidal and M. L. Rodi\~no Montoya to the Universidade Federal de Pernambuco (UFPE). The authors extend their gratitude to this institution for its hospitality and support. This work was partially by the Conselho Nacional de Desenvolvimento Científico e Tecnológico - CNPq (Grant 316121/2023-1), the Funda\c{c}\~ao de Amparo \`a Ci\^encia e Tecnologia do Estado de Pernambuco - FACEPE (Grant APQ-1341-1.02/22) and the Funda\c{c}\~ao de Amparo \`a Pesquisa do Estado de S\~ao Paulo - FAPESP (Grants 2017/10555-0, 2022/08948-2).

\end{document}